\documentclass{SCIYA2017enOL}

\usepackage{epstopdf}
\usepackage{esint}

\DeclareMathOperator{\Ric}{Ric}

\online
\begin{document}

\ensubject{fdsfd}

\ArticleType{ARTICLES}
\Year{2017}
\Month{January}%
\Vol{60}
\No{1}
\BeginPage{1} %
\DOI{10.1007/s11425-000-0000-0}
\ReceiveDate{January 7, 2020}
\AcceptDate{March 11, 2021}

\title[Regularity of IMCF in 3-dimensional AH manifolds]{Regularity of inverse mean curvature flow in asymptotically hyperbolic manifolds with dimension $3$}
{Regularity of inverse mean curvature flow in asymptotically hyperbolic manifolds with dimension $3$}

\author[*]{Yuguang Shi}{{ygshi@math.pku.edu.cn}}
\author[]{Jintian Zhu}{{zhujt@pku.edu.cn}}

\AuthorMark{Shi Y}

\AuthorCitation{Shi Y, Zhu J}

\address[]{Key Laboratory of Pure and Applied Mathematics, School of Mathematical Sciences,\\
 Peking University, Beijing, {\rm 100871}, China}


\abstract{By making use of the nice behavior of Hawking masses of slices of a weak solution of inverse mean curvature flow in  three dimensional asymptotically hyperbolic manifolds, we are able to show that each slice of the flow is star-shaped after a long time, and then we get the regularity of the  weak solution of inverse mean curvature flow in asymptotically hyperbolic manifolds. As an application, we prove that the limit of Hawking mass of the slices of a weak solution of inverse mean curvature flow with any  connected $C^2$-smooth surface as initial data in asymptotically ADS-Schwarzschild manifolds with positive mass is bigger than or equal to the total mass, which is completely different from the situation in asymptotically flat case.}



\maketitle

\section{Introduction}
A Riemannian $3$-manifold $(M^3,g)$ is {\it asymptotically hyperbolic} if it is connected and if there are a bounded open set $U \subset M$ and a diffeomorphism
$$
M \setminus U \cong_x \mathbb R^3 \setminus B_1(0)
$$
such that, in polar coordinates, the metric is in the form
\begin{equation}\label{Eq: AHmetric}
g=\mathrm dr\otimes\mathrm dr+\sinh^2r g_{\mathbb S^2}+Q,
\end{equation}
where
\begin{equation}\label{Eq: errorQdecay}
| Q  |_{\bar g} + | \bar \nabla Q  |_{\bar g} +  |  \bar \nabla^{2} Q |_{\bar g} + |\bar\nabla^3 Q|_{\bar g} = O ( e^{- 3  r} ),
\end{equation}
and  $\bar g$ denotes the hyperbolic metric $ \bar g = dr\otimes dr +\sinh^2 r g_{\mathbb{S}^2}  $.

In this paper, we also require that the scalar curvature of $(M,g)$ satisfies
\begin{equation}\label{Eq: scalarcurvature}
R+6=O(e^{-\alpha r}),\quad\text{\rm for some}\quad\alpha>3.
\end{equation}

In above definition, $M\backslash U$ is called an {\it exterior region} of the asymptotically hyperbolic manifold $(M,g)$.

An asymptotically hyperbolic $3$-manifold $(M,g)$ is called {\it asymptotically Schwarzschild-anti-deSitter} if there is an exterior region such that the metric is like
\begin{equation}\label{def:SadS}
g = dr\otimes dr + \left( \sinh^2 r + \frac{m }{ 3 \sinh r}  \right) g_{\mathbb{S}^2} + Q,
\end{equation}
where
\begin{equation}\label{Eq: AAdSerrorQdecay}
| Q  |_{\bar g} + | \bar \nabla Q  |_{\bar g} +  |  \bar \nabla^{2} Q |_{\bar g} + |\bar\nabla^3 Q|_{\bar g} = O ( e^{- 5  r} ).
\end{equation}
 {For convenience, such an $(M,g)$ will be called an asymptotically ADS-Schwarzschild manifold in our later discussion.}

A smooth solution of {\it inverse mean curvature flow} (IMCF) in an asymptotically hyperbolic manifold $(M,g)$ means a smooth family of surfaces $F:\Sigma\times [0,T)\to M$, which satisfies
\begin{equation}\label{def: IMCF}
\frac{\partial}{\partial t}F(p,t)=H(p,t)^{-1}\nu(p,t),\quad F(\cdot,0)=F_0,
\end{equation}
where $F_0$ is a smooth embedding of $\Sigma$. Denote $\Sigma_t=F(\Sigma,t)$, we also say inverse mean curvature flow $\Sigma_t$ and call $\Sigma_0$ the initial data surface of the flow for convenience.

In order to prove Penrose Inequality, a theory of {weak solutions} of inverse mean curvature flow was developed   in \cite{HI2001} by Huisken and Ilmanen. Namely, they raised up a weak notion of level-set solution for inverse mean curvature flow and obtained the existence for a weak solution with a strictly outer-minimizing $C^2$-smooth initial data surface  in general complete, connected Riemannian $n$-manifolds, which admit a suitable subsolution at the infinity. In particular, a weak solution of inverse mean curvature flow with an outer-minimizing initial data surface always exists in an asymptotically hyperbolic $3$-manifold $(M,g)$.  As mentioned in \cite{HI2001} , a kind of ``jump phenomenon"   may take place in a weak solution of inverse mean curvature flow in asymptotically flat (or hyperbolic) manifolds, and due to this, a smooth solution may not be a weak solution automatically. However, one can show  each  slice of such a weak solution of inverse mean curvature flow is of $C^{1, \alpha}$ and $W^{2,p}$ for some $\alpha \in (0,1)$ and any $p>1$.  Besides \cite{HI2001}, there are several interesting applications of inverse mean curvature flow to various geometric inequalities, see  \cite{BC2014}, \cite{Ch2016},  \cite{BHW2016},  \cite{DN2015}, \cite{GL2009}, \cite{GWW2014}, \cite{W2018}, etc.

In this paper, we  consider the regularity of a weak solution of inverse mean curvature flow, {whose existence was established in \cite{HI2001}},  in an asymptotically hyperbolic manifold $(M, g)$. Namely, we are going to show

\begin{theorem}\label{Thm: maintheorem1}
Let $(M,g)$ be an asymptotically hyperbolic {3-manifold} and $\{\Sigma_t\}_{t\geq 0}$ be a weak solution of inverse mean curvature flow with a  connected $C^2$-smooth initial data surface in $(M,g)$. Then there is a $T_0>0$ such that $\{\Sigma_t\}_{t\geq T_0}$ is smooth.
\end{theorem}

Regularity of {weak solutions} of inverse mean curvature flow  was first established by Huisken and Ilmanen in \cite{HI2008} in Euclidean space $\mathbb R^n$.  Later, it was generalized  to some rotational symmetric manifolds in \cite{LW2017}. In \cite{Ne2010}, A. Neves obtained a precise behavior of the slices of a smooth solution of inverse mean curvature flow in an AdS-Schwarzschild manifold with positive mass by a careful analysis of  various geometric quantities along the flow. However, all of those works in \cite{LW2017} and \cite{Ne2010}  need to assume that  the initial surfaces are star-shaped, which plays a  crucial role in their proofs. Another {difficulty} in investigation of such problems in asymptotically hyperbolic situation is the failure of blow-down arguments which is very effective in asymptotically flat case.

Indeed, star shape is a kind of the first derivative assumption of the surfaces, by this one may get a positive lower bound of the mean curvature of the slices of a smooth solution of inverse mean curvature flow, which can be regarded as  a second derivative condition of the surfaces. Based on this, one can further get higher order estimates of the inverse mean curvature flow. However, {for} general asymptotically hyperbolic manifolds, the notion of star shape of an initial surface seems not to make sense if the surface is not contained in any exterior region.  We observe that, in three dimensional case,  Hawking mass of the slice of a weak solution of inverse mean curvature flow in an asymptotically hyperbolic manifold $(M, g)$ has a uniform lower bound, which then implies that the topology of each slice is sphere after a long time. Furthermore, together with Gauss equations, we get a nice decay estimate of  $L^2$-norm of trace-free part of the second fundamental forms of the slices, which can be regarded as a kind of $W^{2,2}$  a prior estimates of the surfaces. This fact, together with uniform  $C^{1,\alpha}$ estimates obtained in \cite{HI2001}, enables us to show that each slice  of a weak solution of inverse mean curvature flow is automatically star-shaped  when $t \geq T_0$ for some large time $T_0$ (see Corollary \ref{Cor: angleestimatesigmat}).  Unlike \cite{HI2008}, \cite{LW2017}, \cite{Ne2010}, etc,  our ambient manifold $(M^3,g)$ may not be warp-product,  the second fundamental forms of the   slices  are involved in error terms of evolution equations of several geometric quantities, and that makes estimates much more complicated.

For any large time $t>0$, let $\Sigma_t$ be the slices of a weak solution of inverse mean curvature flow in $(M, g)$. As in \cite{HI2008}, we may use mean curvature flow to construct a sequence of smooth surfaces to approximate $\Sigma_t$. Therefore, for any small $s>0$,   we can use $\Sigma_{t,s}$ to denote the slice of mean curvature flow with initial surface $\Sigma_t$, whose detailed construction is left to Lemma \ref{Lem: MFCsigmats}. Note that  $\Sigma_{t,s}$   is  strictly mean convex in $(M, g)$ for small $s$, it can be taken as an initial data surface for a smooth solution of inverse mean curvature flow.    Now, let $\bar\Sigma_{t,s,\tau}$ be the slice of the smooth solution of inverse mean curvature flow with $\Sigma_{t,s}$ as the initial surface at time $\tau>0$.

Unlike in \cite{HI2008}, \cite{LW2017}, \cite{Ne2010}, we use another way rather than investigating evolution equations to obtain a long-time estimate for the star-shape of $\bar\Sigma_{t,s,\tau}$, which guarantees us to get rid of a requirement on the warp-product structure and positive mass of the ambient manifold $(M,g)$. In fact, we observe that the second fundamental form of $\bar\Sigma_{t,s,\tau}$ {enjoys} a uniform upper bounded estimate (see Corollary \ref{Cor: interiorestimateh}), which provides us a uniform $C^{1,\alpha}$ estimates of $\bar\Sigma_{t,s,\tau}$. Notice that $\bar\Sigma_{t,s,\tau}$ also have a uniform lower bound on their Hawking {mass}, after playing the same trick as what we have done for $\Sigma_t$, we can see that $\bar\Sigma_{t,s,\tau}$ is star-shaped (see Proposition \ref{Prop: barsigmatstau}) for all $\tau>0$ when $t$ is large enough. Then by Krylov's regularity theory (see Page 253 \cite{K1987}), we get uniform
 upper bounds of the higher order estimate of second fundamental forms of $\bar\Sigma_{t,s,\tau}$. From the compactness of the weak solution of inverse mean curvature flow, we see that there is a smooth solution for inverse mean curvature flow  with $\Sigma_t$ as the initial data for some large $t$. Therefore, by the uniqueness of the weak solution of inverse mean curvature flow,  we get  the regularity of the inverse mean curvature flow after a large time.

 As an application of Theorem \ref{Thm: maintheorem1}, we can obtain the following result:

 \begin{theorem}\label{maintheorem2}
Let $(M^3, g)$ be an asymptotically ADS-Schwarzschild manifold with  $m>0$ and $\Sigma_t$  be a weak solution of inverse mean curvature flow  with initial data surface $\Sigma$, where $\Sigma$ is any strictly outer-minimizing and connected $C^2$-smooth surface. Then we have
$$
\lim_{t\to+\infty} m_H(\Sigma_t)\geq\frac{m}{2},
$$
with equality if and only if
$$\lim_{t\to+\infty}(\bar r_t-\underline r_t)=0,$$
where
\begin{equation}
\bar r_t:=\max_{\Sigma_t}r,\quad\text{\rm and}\quad\underline r_t:=\min_{\Sigma_t}r.
\end{equation}
\end{theorem}

 Note that the total mass of $(M^3, g)$ is $\frac{m}2$,   due to Theorem \ref{maintheorem2}, we know that it is almost impossible to show the Penrose inequality by inverse mean curvature flow in asymptotically hyperbolic manifolds, which was pointed out in \cite{Ne2010} by considering inverse mean curvature flow with some special initial data. Our arguments of proof Theorem \ref{maintheorem2} are mainly from  \cite{Ne2010}.

The remaining of the paper is organized as follows. In section 2, we give some preliminary lemmas which will  be used later. We show that, for a weak solution of inverse mean curvature flow, the slices $\Sigma_t$ are $L^2$-nearly umbilical spheres for large times $t$. Based on this, we prove that {$\Sigma_t$ is star-shape} for $t$ large enough. In section 3, we prove by Stampacchia iteration process that a smooth solution of inverse mean curvature flow with star-shaped slices $\bar\Sigma_t$ has a lower bound estimate for the mean curvature of $\bar\Sigma_t$, which is independent of the mean curvature lower bound of the initial data surface. Higher order estimates of $\bar\Sigma_t$ will follow from lower bounded mean curvature and bounded second fundamental form by Krylov's regularity theory. Based on this, an extension lemma is introduced at the end of this section. In section 4, we present a proof for our main Theorem \ref{Thm: maintheorem1}. We show that the long-time existence of the approximation flow $\bar\Sigma_{t,s,\tau}$ (see Appendix \ref{section: approximationflow}) for sufficiently large $t$. By {taking} $s\to 0$, we obtain a smooth solution of inverse mean curvature flow, whose slices coincide with those of the weak solution $\Sigma_t$. This gives the smoothness of $\Sigma_t$ for large time $t$. In Section 5, we show Theorem \ref{maintheorem2}.
\section{Preliminary results}

In this section, $(M,g)$ is an asymptotically hyperbolic manifold and $\Sigma_t$ denotes a weak solution of inverse mean curvature flow in $(M,g)$ with a connected $C^2$-smooth initial data.

For any closed surface $\Sigma$, we denote $\nu$ the outward unit normal of $\Sigma$ and $H$ the mean curvature of $\Sigma$ with respect to $\nu$  in $(M,g)$. We use $A$ to represent the second fundamental form of $\Sigma$ and $\mathring A$ the trace-free part of $A$. The $g$-area of $\Sigma$ will be denoted by $A(\Sigma)$.

In order to simplify the statement, we leave some basic properties of a weak solution $\Sigma_t$ to appendix \ref{Append: IMCFproperty}. Here we present some key lemmas for our later proof for the main theorem.

First, we show that the slice $\Sigma_t$ of a weak solution of inverse mean curvature flow is necessarily a $L^2$-nearly umbilical sphere for large time $t$. That is,

\begin{lemma}\label{Lem: nearlyumbilicalsphere sigmat}
There is a $T_0>0$ such that $\Sigma_t$ is topological sphere and
\begin{equation}
\int_{\Sigma_t}\|\mathring A\|^2\,\mathrm d\mu\leq CA(\Sigma_t)^{-\frac{1}{2}}
\end{equation}
for $t\geq T_0$, where $C$ is a universal constant independent of $t$.
\end{lemma}
\begin{proof}
From Lemma \ref{Lem: hawkingmasssigmat}, there is a universal nonnegative constant $\Lambda$ so that $m_H(\Sigma_t)\geq-\Lambda$ for all $t\geq 0$, where $m_H(\Sigma_t)$ is the Hawking mass of $\Sigma_t$, defined by
\begin{equation*}
m_H(\Sigma_t)=\frac{A(\Sigma_t)^{\frac{1}{2}}}{(16\pi)^{\frac{3}{2}}}\left(16\pi-\int_{\Sigma_t}(H^2-4)\,\mathrm d\mu\right).
\end{equation*}
Therefore, we have
\begin{equation}
\int_{\Sigma_t}(H^2-4)\,\mathrm d\mu\leq 16\pi+(16\pi)^{\frac{3}{2}}\Lambda A(\Sigma_t)^{-\frac{1}{2}}.
\end{equation}
Due to Lemma \ref{Lem: pinch sigma t}, for $t$ large enough, $\Sigma_t$ is a surface in the exterior region $M\backslash U$, therefore it can be viewed as a surface in $(M\backslash U,\bar g)$. Denote $\bar H$ and $\mathrm d\bar\mu$ the weak mean curvature and the induced measure of $\Sigma_t$ in $(M\backslash U,\bar g)$. Through direct calculation, we have
\begin{equation}
\int_{\Sigma_t}(\bar H^2-4)\,\mathrm d\bar\mu=\int_{\Sigma_t}(H^2-4)\,\mathrm d\mu+\int_{\Sigma_t}\left(1+\|A\|^2\right) O(e^{-3r})\,\mathrm d\mu.
\end{equation}
Using the uniform bound for second fundamental form of $\Sigma_t$ from Lemma \ref{Lem: C11estimate}, combined with Lemma \ref{Lem: pinch sigma t} and exponential area growth of $\Sigma_t$, we deduce that
\begin{equation}
\int_{\Sigma_t}(\bar H^2-4)\,\mathrm d\bar\mu\leq 16\pi+Ce^{-\frac{t}{2}}\leq 16\pi+CA(\Sigma_t)^{-\frac{1}{2}}.
\end{equation}
Here and in the sequel, $C$ denotes a universal constant independent of $t$.
Since $\Sigma_t$ is a $W^{2,2}$ surface from Lemma \ref{Lem: W 2,p norm sigma t}, by approximation and applying Theorem A in \cite{Marques2013} {(embedded closed surface with positive genus in Euclidean $3$-space has its Willmore energy no less than $2\pi^2$) to $\Sigma_t$ as a surface in Euclidean unit ball}, we can find $T_0$ large so that $\Sigma_t$ is a topological sphere for $t\geq T_0$. Now, combined with the weak Gauss-Bonnet formula from Lemma 5.4 in \cite{HI2001},
\begin{equation}
\int_{\Sigma_t}\|\mathring A\|^2\,\mathrm d\mu=8\pi-4\pi\chi(\Sigma_t)+\int_{\Sigma_t}O(e^{-3r})\mathrm d\mu+32\pi^{\frac{3}{2}}\Lambda A(\Sigma_t)^{-\frac{1}{2}}\leq CA(\Sigma_t)^{-\frac{1}{2}},
\end{equation}
where we have used $\chi(\Sigma_t)=2$ for $t\geq T_0$.
\end{proof}

A sequence of closed surfaces $\Sigma_i$ in $(M,g)$ is called {\it exhaustion} if any compact set of $M$ {is} enclosed by  $\Sigma_i$ for sufficiently large $i$. The following proposition plays a crucial role in our work.

 \begin{proposition}\label{starshapeestimate}
Let $\{\Sigma_i\}$ be a sequence of exhausting closed surfaces in the exterior region $M\backslash U$ with uniform $C^{1,\alpha}$ {bound} for some $0<\alpha<1$ and
\begin{equation}\label{nearlyround1}
\int_{\Sigma_i} \|\mathring {\bar A}\|_{\bar g}^2 d\mu_{\bar g} \to0,\quad\text{\rm as}\quad i\to\infty,
\end{equation}
where $\bar g$ is the hyperbolic metric, $\mathring{\bar A}$ is the second fundamental form of $\Sigma_i$ with respect to $(M\backslash U,\bar g)$. Then for any $\eta >0$, there is {an} $i_0$ which depends only on $\eta$ so that for any $i\geq i_0$, we have
\begin{equation}\label{starshape1}
\left \langle v, \frac{\partial}{\partial r}\right \rangle_{g} \geq 1-\eta.
\end{equation}
\end{proposition}

\begin{remark}
When surfaces $\Sigma_i$ are slices of a weak solution of inverse mean curvature flow, (\ref{nearlyround1}) can be obtained from
$$
\int_{\Sigma_i} \|\mathring { A}\|_g^2 d\mu_{ g} \leq C A(\Sigma_i)^{-\frac12},
$$
combined with Lemma \ref{Lem: pinch sigma t}, Lemma \ref{Lem: C11estimate} and exponential area growth {under} the inverse mean curvature flow.
\end{remark}
\begin{proof}
It is not difficult to see that $\left <v, \frac{\partial}{\partial r}\right >_g$ is very close to $\left <\bar v, \frac{\partial}{\partial r}\right >_{\bar g}$ for sufficiently large $i$, here $\bar v$ is the unit outward normal vector of $\Sigma_i$ with respect to $\bar g$. With this fact in mind, it suffices to show the conclusion of Proposition \ref{starshapeestimate} {in the case of hyperbolic space $\mathbb{H}^3$}. Here and in the sequel, we regard $(M\backslash U,\bar g)$ as $\mathbb{H}^3$.  Now, let $p_i$ be the point in $\Sigma_i$ at which the minimum of $\left <\bar v, \frac{\partial}{\partial r}\right >_{\bar g}$ is achieved, and $T^i:\,\mathbb{H}^3 \to \mathbb{H}^3$ be an isometric transformation with $T^i (p_i)=p$, where $p$ is a fixed point in $\mathbb{H}^3$. Then we are going to show
\begin{equation}\label{starshape2}
\left <T^i _* (\bar v), T^i _* (\frac{\partial}{\partial r})\right >_{\bar g} (p)  \geq 1-\eta.
\end{equation}

Let $o\in \mathbb{H}^3$ {be a point} so that the direction  of the geodesic $\gamma_i$  in  $\mathbb{H}^3$  joining $T^i (o)$ and $p$ is $T^i _* (\frac{\partial}{\partial r})$ at $p$. As $p_i$ diverges to the infinity of $\mathbb{H}^3$, so does $T^i (o)=q_i$.  We adopt ball model for $\mathbb{H}^3$, then $S_i=T^i (\Sigma_i)$ can also be regarded {as} a surface in the unit ball $\mathbb{B}^3$ in $\mathbb{R}^3$, then $p\in S_i$ and $q_i$ is enclosed by $S_i$ and converges to the boundary of $\mathbb{B}^3$. By (\ref{nearlyround1}) and the conformal invariance, we have
\begin{equation}\label{nearlyround3}
\int_{S_i} \|\mathring {\bar{ \bar A}}\|^2 d\mu_0  \to 0,\quad\text{\rm as}\quad i\to\infty,
\end{equation}
where  $ \mathring {\bar{ \bar A}}$  denotes   the trace-free part of the second fundamental forms of $S_i$ in $\mathbb{B}^3$ with the standard Euclidean metric. Together {with these facts and Theorem 1.1} in \cite{LM2005}, we see that  the area of $S_i$ with respect to the Euclidean metric has uniformly lower bound. Again by Theorem 1.1 in \cite{LM2005}, we deduce that, after passing to a subsequence, $S_i$ converges in $C^0$ sense  to the round sphere  $S \subset \mathbb{B}^3  $ which passes  through $p$ and $q=  S\cap \partial\mathbb{B}^3 $, {where $q$ is the limit of $q_i$.} Thus $\gamma_i$ converges to the geodesic line $\gamma$ which joins the infinity point $q$ and $p$, and the direction of $\gamma$ at $p$ is $ \frac{\partial}{\partial r}$. The situation is briefly illustrated by Figure \ref{Figure:1} on next page. Note that $S$ is actually a horosphere of $\mathbb{H}^3$.
\begin{figure}[htbp]
\centering
\includegraphics[width=7cm]{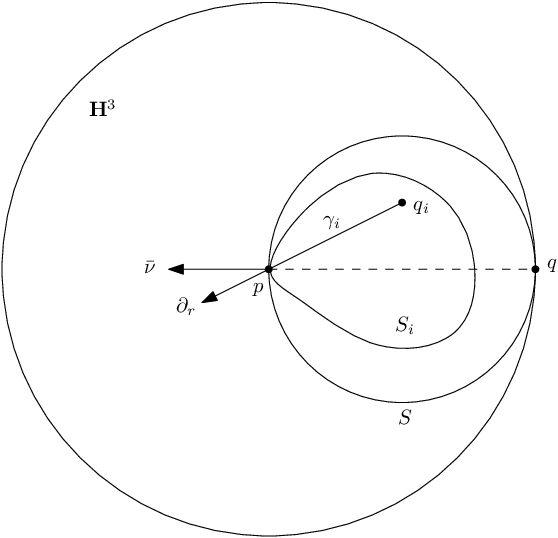}
\caption{Surface $S_i$ in ball model}
\label{Figure:1}
\end{figure}

Next, we are going to show that  the outward unit normal vector of $S_i$ at $p$ with respect to the hyperbolic metric $\bar g$ converges to that of $S$ at $p$ in $C^0$ sense. It is easy to see that $\Sigma_i$ enjoys uniform $C^{1, \alpha}$-estimate with respect to the hyperbolic metric $\bar g$ and so does $S_i$, then by choosing  a subsequence, we may assume $S_i$ locally $C^1$ converges to the limit surface. Due to the above discussion of convergence of surfaces in  $\mathbb{B}^3$, we see that the limit surface is the horophere sphere $S$, which implies

$$
\left <\bar v, \frac{\partial}{\partial r}\right >_{\bar g} (p)  =1.
$$
Therefore (\ref{starshape2}) is true and we get the conclusion of Proposition \ref{starshapeestimate}.
\end{proof}

After applying above proposition to $\Sigma_t$, we have the following corollary.

\begin{corollary}\label{Cor: angleestimatesigmat}
For any $\eta>0$, there is a $T_0=T_0(\eta)$ so that, for $t\geq T_0$, $\Sigma_t$ satisfies
\begin{equation*}
\left\langle\nu,\frac{\partial}{\partial r}\right\rangle_g\geq 1-\eta.
\end{equation*}
\end{corollary}
\begin{proof}
Direct calculation gives
\begin{equation}\label{Eq: starshapesigmat1}
\int_{\Sigma_t}\|\mathring{\bar A}\|^2_{\bar g}\mathrm d\mu_{\bar g}=\int_{\Sigma_t}\|\mathring A\|_g^2\mathrm d\mu_g+\int_{\Sigma_t}\left(1+\|A\|_g^2\right)O(e^{-3r})\mathrm d\mu_g.
\end{equation}
Using Lemma \ref{Lem: pinch sigma t} and the uniform bound for second fundamental form of $\Sigma_t$ from Lemma \ref{Lem: C11estimate}, combined with the exponential growth of area $A(\Sigma_t)=A(\Sigma_0)e^t$, we obtain
\begin{equation}\label{Eq: starshapesigmat2}
\int_{\Sigma_t}\left(1+\|A\|_g^2\right)O(e^{-3r})\mathrm d\mu_g\leq Ce^{-\frac{3}{2}t} A(\Sigma_t)\leq CA(\Sigma_t)^{-\frac{1}{2}}.
\end{equation}
Here and in the sequel, $C$ is a universal constant independent of $t$. Combining (\ref{Eq: starshapesigmat1}), (\ref{Eq: starshapesigmat2}) and Lemma \ref{Lem: nearlyumbilicalsphere sigmat}, we have
\begin{equation}\label{Eq: starshapesigmat3}
\int_{\Sigma_t}\|\mathring{\bar A}\|^2_{\bar g}\mathrm d\mu_{\bar g}\leq CA(\Sigma_t)^{-\frac{1}{2}}.
\end{equation}
Now, the corollary follows from a contradiction argument. Suppose the consequence does not hold, then there exists a sequence of surfaces $\Sigma_{t_i}$ with $t_i\to+\infty$, which satisfies
\begin{equation}\label{Eq: starshapesigmat4}
\langle\frac{\partial}{\partial r},\nu\rangle_g< 1-\eta.
\end{equation}
From (\ref{Eq: starshapesigmat3}) and Lemma \ref{Lem: W 2,p norm sigma t}, we can apply Proposition \ref{starshapeestimate} to surfaces $\Sigma_{t_i}$, which leads to a contradiction to (\ref{Eq: starshapesigmat4}).
\end{proof}

\section{Stampacchia Iteration}
In this section, $(M,g)$ denotes an asymptotically hyperbolic manifold and $\bar\Sigma_t$ denotes a smooth solution of inverse mean curvature flow. With the idea in \cite{HI2008}, we use the Stampacchia iteration process to derive a lower bound for mean curvature of slices $\bar\Sigma_t$, which is independent of the mean curvature bound of the initial data surface. In asymptotically hyperbolic manifolds, the process can be simplified due to the negativity of the Ricci tensor, which was observed in \cite{LW2017}.

\begin{proposition}\label{Prop: lower bound mean curvature}
Given positive constants $\delta_0$, $C_1$ and $C_2$,
there is a universal constant $R_0=R_0(\delta_0,C_2)>0$ so that if $\bar\Sigma_{t}$, $0\leq t< T$, is a smooth inverse mean curvature flow enclosing $B_{R_0}$ with $\langle \nu,\frac{\partial}{\partial r}\rangle\geq \delta_0$  for each slice $\bar\Sigma_t$, and if $\bar\Sigma_0$ satisfies
$\bar r_0-\underline r_0\leq C_1$ and $\max_{\bar\Sigma_0}\|A\|\leq C_2,
$
then
\begin{equation}
H\geq C(\delta_0,C_1,C_2)\min\{1,t^{\frac{1}{2}}\},
\end{equation}
for any $t\in [0, T)$.
\end{proposition}
\begin{proof}
From Lemma \ref{Lem: upper bound mean curvature} and Corollary \ref{Cor: interiorestimateh}, slices $\bar\Sigma_{t}$ satisfies $\|A\|\leq C(\delta_0,C_2)$. Combined with Lemma \ref{Lem: evolutionequation}, we have
\begin{equation*}
\left(\frac{\partial}{\partial t}-\frac{1}{H^2}\Delta\right)w=\frac{\|A\|^2}{H^2}w+\frac{1} {H^2}O(e^{-2r})
\end{equation*}
and
\begin{equation*}
\left(\frac{\partial}{\partial t}-\frac{1}{H^2}\Delta_{\bar\Sigma_t}\right)H^{-1}=\frac{\|A\|^2}{H^2}H^{-1}+ \frac{\Ric(\nu,\nu)}{H^2}H^{-1},
\end{equation*}
where $w=\sinh r\langle \frac{\partial}{\partial r},\nu\rangle$.
We consider the slices $\bar\Sigma_t$ with $t\in[t_0,t_1]$, where $t_0$ and $t_1$ are positive constants to be determined. Denote $v=(t-t_0)^{\frac{1}{2}}H^{-1}w^{-1}$ and $v_k=\max\{v-k,0\}$. Direct calculation shows
\begin{equation}
\begin{split}
\frac{\partial}{\partial t}v\leq \text{\rm div}&\left(\frac{1}{H^2}\nabla v\right)-\frac{1}{H^2}v^{-1}|\nabla v|^2+\frac{1}{2}(t-t_0)^{-1}v\\
&\qquad+(t-t_0)^{-1}v^3w^2\left(\Ric(\nu,\nu)+C(\delta_0,C_2)w^{-1}e^{-2r} \right),\quad\forall\, t_0<t\leq t_1.
\end{split}
\end{equation}
Using the fact that $\delta_0\sinh r\leq w\leq \sinh r$, by taking $R_0=R_0(\delta_0,C_2)$ sufficiently large, we can guarantee that
\begin{equation*}
\Ric(\nu,\nu)+C(\delta_0,C_2)w^{-1}e^{-2r}\leq -1.
\end{equation*}
Furthermore, there holds
\begin{equation}
\begin{split}
\frac{\mathrm d}{\mathrm dt}\int_{\bar\Sigma_t}v_k^2\mathrm d\mu\leq (t-t_0)^{-1}\int_{\Omega_t(k)}v_k&v\,\mathrm d\mu+\int_{\Omega_t(k)}v_k^2\,\mathrm d\mu\\
&-2(t-t_0)^{-1}\delta_0^2\sinh^2\underline r_t\int_{\Omega_t(k)}v_kv^3\,\mathrm d\mu,
\end{split}
\end{equation}
where $\Omega_t(k)=\{x\in\bar\Sigma_t:v\geq k\}$. Taking
$
k=\delta_0^{-1}\sinh^{-1}\underline r_{t_0}\max\{t_1-t_0,1\},
$
then
\begin{equation}
\frac{\mathrm d}{\mathrm dt}\int_{\bar\Sigma_t}v_k^2\,\mathrm d\mu\leq 0.
\end{equation}
Notice that $v_k\equiv 0$ on $\bar\Sigma_{t_0}$, we have $v_k\equiv 0$ for $\bar\Sigma_t$ with $t\in[t_0,t_1]$, which implies $v\leq k$ on these slices.

For the choice for $t_0$ and $t_1$, we divide into two cases. If $0< t_1\leq 2$, we choose $t_0=t_1/2$. Otherwise, $t_0=t_1-1$. In both cases, $k$ can be taken to be $\delta_0^{-1}\sinh^{-1}\underline r_{t_0}$. Combined the fact $v\leq k$ on $\bar\Sigma_{t_1}$ with the definition of $v$, using also the radii estimate in Lemma \ref{Lem: pinch lemma}, we conclude that on surface $\Sigma_{t_1}$
\begin{equation}
H\geq C(\delta_0,C_1,C_2)\min\{t_1^{\frac{1}{2}},1\},
\end{equation}
which completes the proof.
\end{proof}

\begin{lemma}\label{Lem: extension IMCF}
Let $\bar\Sigma_t$, $0\leq t<T<+\infty$, be a smooth inverse mean curvature flow in $(M,g)$. If there exist $c_0$ and $c_1$ such that $H\geq c_0>0$ and $\|A\|\leq c_1$, then $\bar\Sigma_t$ can be extended beyond the time $T$.
\end{lemma}
\begin{proof}
The theorem {on} Page 253 of \cite{K1987} guarantees higher regularity of the solution, therefore $\Sigma_t$ converges smoothly to a smooth limit surface $\Sigma_T$ with mean curvature $H\geq c_0$. Then the short time existence of solution to (\ref{def: IMCF}) in case of smooth initial data surface with positive mean curvature yields the desired extension.
\end{proof}

\section{Proof for Theorem \ref{Thm: maintheorem1}}
In this section, $(M,g)$ is {assumed to be} an asymptotically hyperbolic manifold and $\Sigma_t$ is a weak solution of inverse mean curvature flow in $(M,g)$.
For large $t$ and small $s$, $\Sigma_{t,s}$ represents the slice of the mean curvature flow with initial data surface $\Sigma_t$ at time $s>0$. Since $\Sigma_{t,s}$ is smooth with strictly positive mean curvature,  $\Sigma_{t,s}$ can be taken as an initial data surface of a smooth solution of inverse mean curvature flow. Denote the slice of the solution at time $\tau>0$ by $\bar\Sigma_{t,s,\tau}$ and $\tau_0(t,s)$ the maximum existence time for $\bar\Sigma_{t,s,\tau}$.

\begin{proposition}\label{Prop: barsigmatstau}
For any $0<\delta_0<1$, there is a $T_0=T_0(\delta_0)>0$ such that for any $\Sigma_{t,s}$ with $t\geq T_0$, the smooth solution $\bar\Sigma_{t,s,\tau}$ of inverse mean  curvature flow with initial data surface $\Sigma_{t,s}$ exists for all $\tau>0$ and satisfies
\begin{equation}
\langle\frac{\partial}{\partial r},\nu\rangle\geq\delta_0.
\end{equation}
\end{proposition}
\begin{proof}
Define
\begin{equation}\label{Eq: tau1ts}
\tau_1(t,s):=\sup\left\{\tau'>0: \,\bar\Sigma_{t,s,\tau}\,\,\text{\rm exists and satisfies }\langle \partial_r,\nu\rangle\geq\delta_0\,\,\text{\rm for }\tau\in[0,\tau']\right\}.
\end{equation}
From Corollary \ref{Cor: angleestimatesigmat} and the $C^{1,\alpha}$ convergence of $\Sigma_{t,s}$ to $\Sigma_t$ in Lemma \ref{Lem: MFCsigmats}, there exists a $T_0>0$ such that above definition makes sense for $\bar\Sigma_{t,s,\tau}$ with $t\geq T_0$. Possibly increasing the value of $T_0$, we are going to show that $\tau_1(t,s)=+\infty$ for $t\geq T_0$. It follows from an argument by contradiction. Assume that the consequence is not true, then we can find $t_i\to\infty$ and $s_i>0$ such that $\tau_1(t_i,s_i)<\infty$. Denote
\begin{equation}
\mathcal S:=\left\{\bar\Sigma_{t_i,s_i,\tau}:\,t_i\geq T_0,\,0\leq\tau <\tau_1(t_i,s_i)\right\}.
\end{equation}
By Lemma \ref{Lem: MFCsigmats} and Corollary \ref{Cor: interiorestimateh}, there exists $i_0$ such that any surface $\bar\Sigma_{t_i,s_i,\tau}$ in $\mathcal S$ with $i\geq i_0$ satisfies $\|A_{i,\tau}\|\leq C$, where $C$ is a universal constant independent of $i$ and $\tau$. In particular, such $\bar\Sigma_{t_i,s_i,\tau}$ has locally uniform $C^{1,\alpha}$ estimate. Combined with Lemma \ref{Lem: MFCsigmats}, Lemma \ref{Lem: sigmatsestimates}, Proposition \ref{Prop: lower bound mean curvature} and Lemma \ref{Lem: extension IMCF}, we can also assume that $\bar\Sigma_{t_i,s_i,\tau}$ with $i\geq i_0$ exists on $[0,\tau_1(t_i,s_i)+\epsilon_i)$, where $\epsilon_i$ are positive constants depending on $i$.
Notice that $\bar\Sigma_{t_i,s_i,\tau}$ is topological sphere due to the fact that $\Sigma_t$ is topological sphere from Lemma \ref{Lem: nearlyumbilicalsphere sigmat}, combined with Lemma \ref{Lem: sigmatstauestimates}, we have
\begin{equation}
\begin{split}
\int_{\bar\Sigma_{t_i,s_i,\tau}} \|\mathring{\bar A}_{i,\tau}\|_{\bar g}^2\mathrm d \mu_{\bar g}&=\int_{\bar\Sigma_{t_i,s_i,\tau}} \|\mathring A_{i,\tau}\|_g^2\mathrm d \mu_g +\int_{\bar\Sigma_{t_i,s_i,\tau}} (1+\|A_{i,\tau}\|_g^2)O(e^{-3r})\mathrm d\mu_g\\
&=8\pi-4\pi\chi(\bar\Sigma_{t_i,s_i,\tau})+32\pi^{\frac{3}{2}}\Lambda A(\bar\Sigma_{t_i,s_i,\tau})^{-\frac{1}{2}}+\int_{\bar\Sigma_{t_i,s_i,\tau}} O(e^{-3r})\mathrm d\mu_g\\
&\leq CA(\Sigma_{t_i,s_i,\tau})^ {-\frac{1}{2}}\to 0,\quad\text{as}\quad i\to\infty,
\end{split}
\end{equation}
where $C$ is a universal constant. Here we remind readers that we use the fact $\|A_{i,\tau}\|\leq C$ in the second line and use the radii estimate from Lemma \ref{Eq: sigmatstauestimates1} and exponential growth of area to handle the last error term.
Applying Proposition \ref{starshapeestimate} to $\bar\Sigma_{t_i,s_i,\tau}$, fixing a $\delta_1\in(\delta_0,1)$, possibly raising the value of $i_0$, we conclude that any surface $\bar\Sigma_{t_i,s_i,\tau}\in\mathcal S$ with $i\geq i_0$ satisfies $\langle\partial_r,\nu\rangle\geq \delta_1$. Since $\bar\Sigma_{t_i,s_i,\tau}$ is a smooth solution of inverse mean curvature flow, this implies that there {exists a positive constant} $\epsilon_i'<\epsilon_i$ such that for $i\geq i_0$
\begin{equation}
\left\langle\frac{\partial}{\partial r},\nu\right\rangle\geq\delta_0\quad\text{\rm for}\quad\bar\Sigma_{t_i,s_i,\tau}\quad\text{\rm with}\quad0\leq\tau\leq\tau_1(t_i,s_i)+\epsilon_i',
\end{equation}
which contradicts to (\ref{Eq: tau1ts}) and this completes the proof.
\end{proof}

We now present the proof for Theorem \ref{Thm: maintheorem1}
\begin{proof}[Proof for Theorem \ref{Thm: maintheorem1}]
Combined Lemma \ref{Lem: sigmatsestimates} and Proposition \ref{Prop: barsigmatstau}, we can take $T_0$ large enough such that Proposition \ref{Prop: lower bound mean curvature} is valid for $\bar\Sigma_{T_0,s,\tau}$. Therefore, for any $0<\tau'<\tau''$, we have $H\geq C(\tau',\tau'')$ for $\bar\Sigma_{T_0,s,\tau}$ with $\tau\in[\tau',\tau'']$, where $C(\tau',\tau'')$ is a universal constant independent of $s$ and $\tau$. Also, from Corollary \ref{Cor: interiorestimateh}, such $\bar\Sigma_{T_0,s,\tau}$ satisfies $\|A\|\leq C$ for {an} absolute constant $C$. Using these and applying Krylov's regularity theory to $\bar\Sigma_{T_0,s,\tau}$, all resulting higher regular estimates are uniform in $s$ for $\tau\in[\tau',\tau'']$. Therefore, $\Sigma_{T_0,s,\tau}$ converges locally and smoothly to a limit inverse mean curvature flow $\bar\Sigma_{T_0,\tau}$, $\tau>0$, as $s\to 0$.

We claim that slices $\bar\Sigma_{T_0,\tau}$ must coincide with the slices $\Sigma_{T_0+\tau}$ of the weak solution of inverse mean curvature flow for all $\tau>0$. By the definition of a weak solution of inverse mean curvature flow in \cite{HI2001}, we assume that $\Sigma_t$ is a level set solution with respect to a Lipschitz function $u$. Define
\begin{equation}
\bar u(x):=\left\{
\begin{array}{cc}
\tau,&x\in\bar\Sigma_{T_0,\tau};\\
0,&\text{\rm otherwise},
\end{array}\right.
\end{equation}
then it suffices to show that $\bar u=(u-T_0)_+$, where $(u-T_0)_+$ represents the nonnegative part of the function $u-T_0$. This {follows} from a comparison between $\Sigma_t$ and $\bar\Sigma_{T_0,\tau}$. From uniformly bounded mean curvature for $\Sigma_{T_0,s}$ in Lemma \ref{Lem: MFCsigmats}, using Lemma \ref{Lem: upper bound mean curvature}, combined with a uniform choice of $\sigma(x)$ by the asymptotically hyperbolic property, $\bar\Sigma_{T_0,s,\tau}$ with $\tau>0$ satisfies $H\leq C$. Here and in the sequel, {$C$ is always denoted to be a} universal constant independent of $s$ and $\tau$, while the meaning may vary from line to line. Investigating the distance between $\bar\Sigma_{T_0,s,\tau}$ and $\Sigma_{T_0,s}$, we have $\text{\rm dist}(\bar\Sigma_{T_0,s,\tau}, \Sigma_{T_0,s})\geq C\tau$. Due to the fact that $\bar\Sigma_{T_0,s,\tau}$ encloses $\Sigma_{T_0,s}$ and $C^{1,\alpha}$ convergence from $\Sigma_{T_0,s}$ to $\Sigma_{T_0}$ in Lemma \ref{Lem: MFCsigmats}, {taking} $s\to0$, we know that $\bar\Sigma_{T_0,\tau}$ encloses $\Sigma_{T_0}$ for any $\tau>0$. From Lemma 2.3 in \cite{HI2001}, {fix} any $\tau_0>0$, $\{\Sigma_{T_0,\tau}\}_{\tau \geq\tau_0}$ is a weak solution of inverse mean curvature flow as level sets of the function $(\bar u-\tau_0)_+$. Compared {with} the weak solution $\{\Sigma_{t}\}_{t\geq T_0}$, applying Lemma 2.2 in \cite{HI2001}, we obtain $(u-T_0)_+\geq (\bar u-\tau_0)_+$. {Taking} $\tau_0\to 0^+$, we see $(u-T_0)_+\geq\bar u$. For {the other} direction, notice that $\bar\Sigma_{T_0,s,\tau}$ with $0<\tau\leq 1$ satisfies $H\geq C\tau^{\frac{1}{2}}$ from Proposition \ref{Prop: lower bound mean curvature}, by investigating the farthest distance between $\bar\Sigma_{T_0,s,\tau}$ and $\Sigma_{T_0,s}$, we conclude that
$$\sup_{p\in\bar\Sigma_{ T_0,s,\tau}} \text{\rm dist}(p,   \Sigma_{T_0,s})\leq C\tau^{\frac{1}{2}}.
$$
{Taking} $s\to 0$, for any $t_0>0$, there exists $\tau_0$ such that $\bar\Sigma_{T_0,\tau}$ with $\tau\leq\tau_0$ is enclosed by $\Sigma_{T_0+t_0}$. Using Lemma 2.2 in \cite{HI2001} again, we have $(u-T_0-t_0)_+\leq (\bar u-\tau)_+$ for any $0<\tau<\tau_0$. {Taking} $\tau\to0$, and then $t_0\to0$, we get $(u-T_0)_+\leq \bar u$. Therefore, $\bar u=(u-T_0)_+$.
\end{proof}

\section{The limit of Hawking masses of slices along the IMCF}\label{section: limithawkingmass}

Let $(M,g)$ be an asymptotically ADS-Schwarzschild manifold with positive mass $\frac{m}2$ and $\Sigma_t$ be a weak solution of inverse mean curvature flow with connected $C^2$-smooth initial data. If we choose $T_0$ large enough, surfaces $\Sigma_t$ with $t\geq T_0$ will be contained in the exterior region $M\backslash U$. Due to the fact that $\Sigma_t$ is star-shaped for $t\geq T_0$, we can view these surfaces as radical graphs over $\mathbb S^2$ in the polar {coordinates}. That is, we write
\begin{equation}
\Sigma_t=\left\{\left(\hat r_t+f_t(\theta),\theta\right):\theta\in\mathbf S^2\right\},
\end{equation}
where $\hat r_t$ is the area radius such that $A(\Sigma_t)=4\pi\sinh^2\hat r_t$.

In the following, we consider the asymptotically ADS-Schwarzschild metric in the form
\begin{equation}
g=\mathrm dr^2+\left(\sinh^2r+\frac{m}{3\sinh r}\right)g_{\mathbb S^2}+Q
\end{equation}
with
\begin{equation}{\label{Eq: errordecayuptolorder}}
\sum_{i=0}^l|\bar\nabla^i Q|_{\bar g}=O(e^{-5r}),\quad l\geq 3.
\end{equation}

We want to show the following result:
\begin{theorem}\label{Thm: main2}
As $t\to+\infty$, functions $f_t$ converge to a $C^{k-1,\alpha}$ function $f$ on $\mathbb S^2$ in $C^0(\mathbb S^2)$ sense, where $k=\min\{5,l+1\}$. Furthermore, the Hawking mass of $\Sigma_t$ satisfies
\begin{equation}
\lim_{t\to+\infty}m_H(\Sigma_t)=\frac{m}{2}\left(\fint_{\mathbb S^2}e^{2f}\mathrm d\mu_{\mathbb S^2}\right)^{\frac{1}{2}}\fint_{\mathbb S^2}e^{-f}\mathrm d\mu_{\mathbb S^2}.
\end{equation}
\end{theorem}

From this, it is clear that we have the following corollary  which is Theorem \ref{maintheorem2}:
\begin{corollary}
The Hawking mass of $\Sigma_t$ satisfies
$$
\lim_{t\to+\infty} m_H(\Sigma_t)\geq\frac{m}{2},
$$
with equality if and only if
$$\lim_{t\to+\infty}(\bar r_t-\underline r_t)=0.$$
\end{corollary}
\begin{proof}
From H\"older's inequality, we have
\begin{equation*}
1=\fint_{\mathbb S^2}1\mathrm d\mu_{\mathbb S^2}\leq \left(\fint_{\mathbb S^2}e^{2f}\mathrm d\mu_{\mathbb S^2}\right)^{\frac{1}{3}}\left(\fint_{\mathbb S^2}e^{-f}\mathrm d\mu_{\mathbb S^2}\right)^{\frac{2}{3}}.
\end{equation*}
The equality holds if and only if $f$ is a constant function.
\end{proof}

Since our following argument strongly relies on  \cite{Ne2010}, let us sketch the corresponding part in \cite{Ne2010} first. Although only exact ADS-Schwarzschild manifolds are under consideration in \cite{Ne2010}, many results are still valid in asymptotically ADS-Schwarzschild case. First, Proposition 2.1 in \cite{Ne2010} is true for asymptotically ADS-Schwarzschild manifolds from the exactly same calculations. Also, arguments in Lemma 3.3 and Lemma 3.4 of \cite{Ne2010} work well for asymptotically ADS-Schwarzschild manifolds. The induction process in Lemma 3.6 of \cite{Ne2010} can be applied as well. However, slight difference appears in the tensor $B=\Ric(\cdot,\nu)$ defined in Lemma 3.6 between these two cases. In ADS-Schwarzschild manifolds, rotational symmetry guarantees that the components of tensor $B$ in a local coordinate can be written as functions $F_j(r,\nabla r)$, which depend only on $r$ and $\nabla r$. While in asymptotically ADS-Schwarzschild case, the component function $F_j=F_j(r,\theta,\nabla r)$ also depends on the sphere parameter $\theta$. In the process for higher order estimate, the derivative along $\theta$-direction will impose a restriction on the highest decay order of $\|\nabla^i A\|$. In our case, the restriction for the decay order of $\|\nabla^i A\|$ is $O(e^{-5r})$, due to the non-rotational symmetry part $Q$. Therefore, by the same discussion on Page 214-217 in \cite{Ne2010}, we {obtain} the same result as in Lemma 3.5 \cite{Ne2010}, except for a restriction on the value of $n$ due to the highest decay order $O(e^{-5r})$ of the differential of $Q$ up to order $l$. In fact, we can obtain
\begin{equation}\label{Eq: fthighorderestimate}
A(\Sigma_0)^n|\nabla^n f_t|^2\leq Ce^{-nt},\quad n=1,2,\ldots,k,
\end{equation}
and
\begin{equation}
A(\Sigma_0)^{n+2}|\nabla^nA|^2\leq Ce^{-(n+2)t},\quad n=1,2,\ldots,k-2,
\end{equation}
where $k=\min\{5,l+1\}$.

In order to apply results in \cite{Ne2010}, we verify the Hypothesis $(H)$ given in section 3 of \cite{Ne2010} for some slice $\Sigma_{t_1}$. For this purpose, we first derive a local uniform $C^{2,\alpha}$ estimate for $\Sigma_t$ from Krylov's regularity theory. Through a similar argument in the spirit of Proposition \ref{starshapeestimate}, we then verify the Hypothesis $(H)$ for some slice $\Sigma_{t_1}$.

\begin{lemma}
Let $T_0$ as in Theorem \ref{Thm: maintheorem1}. There is a $T_1>T_0$ so that $\Sigma_t$ has local uniform $C^{2,\alpha}$-estimate for all $t\geq T_1$.
\end{lemma}
\begin{proof}
From Proposition \ref{Prop: lower bound mean curvature} and Theorem \ref{Thm: maintheorem1}, $\Sigma_{T_0+t}$ satisfies
\begin{equation*}
H\geq C\min\{t^{\frac{1}{2}},1\},
\end{equation*}
where $C$ is a universal constant. Therefore, we can take $T_1>T_0$ such that $\Sigma_t$ with $t\geq T_1$ satisfies uniformly lower bounded mean curvature. Combined with uniformly bounded second fundamental form for $\Sigma_t$, using the theorem in Page 253 of  \cite{K1987}, we conclude that $\Sigma_t$ with $t\geq T_1$ satisfies locally uniform $C^{2,\alpha}$ estimates.
\end{proof}

Then we can obtain estimates for mean curvature and the trace-free part of the second fundamental form.

\begin{lemma}\label{Lem: HandoA}
For any $\epsilon>0$, there is a $t_1\geq T_1$ such that $\Sigma_{t_1}$ satisfies
\begin{equation}
|H-2|\leq \epsilon\quad\text{\rm and}\quad \|\mathring A\|\leq \epsilon.
\end{equation}
\end{lemma}
\begin{proof}
We only prove the first estimate, since the second one follows from a similar argument. Assume that the estimate does not hold, then there exists $t_i\to+\infty$ and $p_i\in\Sigma_{t_i}$ such that $|H(p_i)-2|\geq \epsilon$. We regard $(M,\bar g)$ as $\mathbb H^3$ and view $\Sigma_i$ as surfaces in $\mathbb H^3$. Let $p$ be a fixed point in $\mathbb H^3$ and $T^i:\mathbb H^3\to\mathbf H^3$ be isometric transformations of $\mathbb H^3$ such that $T^i(p_i)=p$. Denote $S_i=T^i(\Sigma_{t_i})$, after adapting the ball model for $\mathbb H^3$, we can view $S_i$ as surfaces in the Euclidean ball $\mathbb B^3$. As in Proposition \ref{starshapeestimate}, we conclude that $S_i$ converge in $C^0$ sense to a round sphere which passes $p$ and tangent to $\partial\mathbf B^3$ at some point $q$. Denote $g_i=\left((T^i)^{-1}\right)^*g$, then $g_i$ converges in $C^{2,\alpha}_{loc}$ sense to $\bar g$. Since $\Sigma_{t_i}$ has local uniform $C^{2,\alpha}$ estimate in $(M,g)$, possibly passing to a subsequence, $S_i=T^i(\Sigma_{t_i})$ converges to a horosphere in $\mathbb H^3$ in $C^{2,\beta}_{loc}$ sense, where $\beta<\alpha$. In particular, we see that $H(p_i)$ converges to 2, which leads to a contradiction.
\end{proof}

Now, we present the proof for Theorem \ref{Thm: main2}.
\begin{proof}[Proof for Theorem \ref{Thm: main2}]
From Lemma \ref{Lem: HandoA}, we can choose a surface $\Sigma_{t_1}$ such that $\Sigma_{t_1}$ satisfies the hypothesis $(H)$ in section 3 of \cite{Ne2010}. From Lemma 3.3, Lemma 3.4 and the calculation on Page 218 in \cite{Ne2010}, we know that $f_t$ converges to a function $f$ on $\mathbb S^2$ in $C^0(\mathbb S^2)$ sense. By (\ref{Eq: fthighorderestimate}), we see that $f$ is a $C^{k-1,\alpha}$ function on $\mathbb S^2$. Since $f_t$ has uniformly bounded $C^2$-norm on $\mathbb S^2$ and $f_t$ converges to $f$ in $C^{0}(\mathbb S^2)$ sense, similar calculation as in Proposition 2.1 \cite{Ne2010} shows that
\begin{equation*}
\lim_{t\to+\infty}m_H(\Sigma_t)=\frac{m}{2}\left(\fint_{\mathbb S^2}e^{2f}\mathrm d\mu_{\mathbb S^2}\right)^{\frac{1}{2}}\fint_{\mathbb S^2}e^{-f}\mathrm d\mu_{\mathbb S^2}.
\end{equation*}
\end{proof}

\begin{appendix}
\section{Evolution equations under IMCF}
In this section, $(M,g)$ is an asymptotically hyperbolic manifold and $\bar\Sigma_t$ is a smooth solution of inverse mean curvature flow in $(M,g)$. We calculate evolution equations of various quantities under inverse mean curvature flow $\bar\Sigma_t$ as following:
\begin{lemma}\label{Lem: evolutionequation}
Let $F=\sinh r\frac{\partial}{\partial r}$ and $w=\langle F,\nu\rangle$, then
\begin{equation}\label{Eq: evolutionequationw}
\left(\frac{\partial}{\partial t}-\frac{1}{H^2}\Delta_{\bar\Sigma_t}\right)w=\frac{\|A\|^2}{H^2}w+\frac{1}{H^2} \left(O(e^{-2r})+O(\|A\|e^{-2r})\right).
\end{equation}
The evolution equation for mean curvature is
\begin{equation}\label{Eq: evolutionequationH}
\frac{\partial H}{\partial t}=\frac{1}{H^2}\Delta_{\bar\Sigma_t} H-\frac{2}{H^3}|\nabla_{\bar\Sigma_t} H|^2-\frac{\Ric(\nu,\nu)+\|A\|^2}{H}.
\end{equation}
In particular,
\begin{equation}\label{Eq: evolutionequationH-1}
\left(\frac{\partial}{\partial t}-\frac{1}{H^2}\Delta_{\bar\Sigma_t}\right)H^{-1}=\frac{\|A^2\|}{H^2}H^{-1}+ \frac{\Ric(\nu,\nu)}{H^2}H^{-1}.
\end{equation}
\end{lemma}
\begin{proof}
Let $(x^1,x^2)$ be an orthogonal coordinate system on $\mathbb S^2$ and denote $x^0=r$, then $(x^0,x^1,x^2)$ is a coordinate system for $(M,g)$. Under this coordinate system, the metric is
\begin{equation*}
g_{0i}=\delta_{0i},\quad g_{\alpha\beta}=\sinh^2r g_{\mathbb S^2,\alpha\beta}+Q_{\alpha\beta}.
\end{equation*}
Here and in the sequel, we use $i,j$ to denote indices from 0 to 2 and $\alpha,\beta$ to denote indices from 1 to 2. It is easy to calculate
\begin{equation*}
(\text{\rm Hess}\,r)_{\alpha\beta}=\Gamma_{\alpha\beta}^0=-\sinh r\cosh rg_{\mathbb S^2,\alpha\beta}-\frac{1}{2}\partial_r Q_{\alpha\beta}.
\end{equation*}
From this, we calculate further that
\begin{equation}\label{Eq: evolutionequation1}
\langle\nabla_vF,w\rangle=\cosh r\langle v,w\rangle-\cosh rQ(v, w)+\frac{1}{2}\sinh r(\partial_rQ)(v,w),
\end{equation}
where $\nabla$ denotes the covariant derivative of $(M,g)$ and $v$ and $w$ are arbitrary vectors. Therefore,
\begin{equation}\label{Eq: evolutionequation2}
\frac{\partial}{\partial t}w=\langle\nabla_{\partial_t}F,\nu\rangle+\langle F,\nabla_{\partial_t}\nu\rangle=H^{-1}\cosh r+H^{-1}O(e^{-2r})+\frac{1}{H^2}\langle F,\nabla_{\bar\Sigma_t}H\rangle.
\end{equation}
Take $(y^1,y^2)$ to be a normal coordinate system on $\bar\Sigma_t$ around a point $p$, using (\ref{Eq: evolutionequation1}), we also obtain at the point $p$ that
\begin{equation}\label{Eq: evolutionequation3}
\begin{split}
\Delta_{\bar\Sigma_t}w&=\partial_\alpha\langle F_{\alpha},\nu\rangle+\langle F_\alpha,\nu_\alpha\rangle+\langle F,\nu_{\alpha\alpha}\rangle\\
&=O(e^{-2r})+O(\|A\|e^{-2r})+H\cosh r+\langle F,\nabla_{\bar\Sigma_t}H\rangle+\Ric(F^T,\nu)-\|A\|^2\langle F,\nu\rangle.
\end{split}
\end{equation}
Now, equation (\ref{Eq: evolutionequationw}) follows from (\ref{Eq: evolutionequation2}) and (\ref{Eq: evolutionequation3}).

Equation (\ref{Eq: evolutionequationH}) comes from Ricatti equation.
\end{proof}

\section{Basic Facts for weak and smooth IMCFs}\label{Append: IMCFproperty}
In this section, $(M,g)$ is an asymptotically hyperbolic manifold and $\Sigma_t$ denotes a weak solution of inverse mean curvature flow with a connected $C^2$-smooth initial data surface and $\bar\Sigma_t$ denotes a smooth solution of inverse mean curvature flow in $(M,g)$.

For any closed surface $\Sigma$ contained in the exterior region, we define the {\it outer radii} $\bar r$ and the {\it inner radii} $\underline r$ of $\Sigma$ by
\begin{equation}\label{maxminradius}
\bar r:=\max_{\Sigma}r \quad\text{\rm and}\quad\underline r:=\min_{\Sigma}r.
\end{equation}
In the following, we will use the notion $\bar r_t$ and $\underline r_t$ to represent the outer radii and the inner radii of $\Sigma_t$ or $\bar\Sigma_t$.

The first two lemmas are radii estimates for smooth or weak solutions of inverse mean curvature flow in $(M,g)$.

\begin{lemma}\label{Lem: pinch lemma}
There are positive constants $R_0$ and $C_0$ so that if $\bar\Sigma_t$ is a smooth or weak inverse mean curvature flow enclosing $B_{R_0}$, then we have
\begin{equation}
\underline r_0+\frac{1}{2}t-C_0\leq\underline r_t\leq\bar r_t\leq \bar r_0+\frac{1}{2}t+C_0,
\end{equation}
where $\bar r_t$ and $\underline r_t$ are the outer and inner radii of $\bar\Sigma_t$.
\end{lemma}

\begin{proof}
We construct spherical sub-solutions and super-solutions as barriers to get the desired lower and upper bounds. Let $S_{\rho_t}$ be a smooth family of expanding spheres. Through direct calculation, it is easy to see
\begin{equation*}
H_{\rho_t}=\frac{2\cosh\rho_t}{\sinh\rho_t} +O(e^{-3\rho_t}).
\end{equation*}
Choosing $R_0$ large enough, we have
\begin{equation*}
\frac{2\cosh\rho_t}{\sinh\rho_t}-2e^{-2\rho_t}\leq H_{\rho_t}\leq \frac{2\cosh\rho_t}{\sinh\rho_t}+2e^{-2\rho_t},\quad\rho_t\geq R_0.
\end{equation*}
This means that the ordinary differential equation
\begin{equation}
\frac{\mathrm d\rho_t^\pm}{\mathrm d t}=\left(\frac{2\cosh\rho_t^\pm}{\sinh\rho_t^\pm}\pm 2e^{-2\rho_t^\pm}\right)^{-1},\quad\rho_0^\pm\geq R_0,
\end{equation}
gives a subsolution $S_{\rho^+_t}$ and a supersolution $S_{\rho^-_t}$. Taking integral, we obtain
\begin{equation*}
\ln\sinh\rho_t^{\pm}-\ln\sinh\rho_0^\pm\pm\left(e^{-2\rho_0^\pm}- e^{-2\rho_t^\pm}\right)=\frac{1}{2}t,\quad\rho_0^\pm\geq R_0.
\end{equation*}
This implies that $|\rho_t^\pm-\rho_0^\pm-\frac{t}{2}|\leq C_0$ for some universal constant $C_0$ independent of $\rho_0^\pm$. Now, the desired result follows easily from the comparison principle for inverse mean curvature flow.
\end{proof}

\begin{lemma}\label{Lem: pinch sigma t}
Let $\Sigma_t$ be a weak solution of inverse mean curvature flow, then there are positive constants $T_0$ and $C_0$ so that any slice $\Sigma_t$ with $t\geq T_0$ satisfies
\begin{equation}
\underline r_0+\frac{1}{2}t-C_0\leq\underline r_t\leq\bar r_t\leq \bar r_0+\frac{1}{2}t+C_0,
\end{equation}
\end{lemma}

\begin{proof}
Let $R_0$ be in Lemma \ref{Lem: pinch lemma}. Since a weak solution of inverse mean curvature flow is a level set solution with respect to a proper, locally Lipschitz function $u$ by Theorem 3.1 in \cite{HI2001}, there exists a $T_0$ such that $\Sigma_t$, $t\geq T_0$, encloses $B_{R_0}$. Denote $\bar r_{T_0}$ and $\underline r_{T_0}$ to be the outer and inner radii of $\Sigma_{T_0}$, then it suffices to show that the surface $\Sigma_{T_0+t'}$ always stays between the spheres $S_{\rho^+_{t'}}$ and $S_{\rho^-_{t'}}$ with $\rho^+_0=\underline r_{T_0}$ and $\rho^-_{0}=\bar r_{T_0}$ for all $t'\geq 0$.
In the following, we are going to prove that $\Sigma_{T_0+t'}$ is inside $S_{\rho^-_{t'}}$ for all $t'\geq 0$. {The other case} can be proved in a similar way.

Define
\begin{equation}
\mathcal S:=\left\{t'\geq 0:\Sigma_{T_0+s}\text{ \rm is inside }S_{\rho^-_{s}}\text{ \rm for all }s\in[0,t']\right\},
\end{equation}
we show that $\mathcal S$ is a non-empty, relatively closed and open subset of $[0,+\infty)$.

Obviously, $0\in\mathcal S$, so $\mathcal S$ is non-empty. Due to the lower semi-continuity of $\Sigma_t$ and the continuity of $S^-_{\rho_{t'}}$,  $\mathcal S$ must be closed. For the openness, we show that if $t_0\in\mathcal S$, there exists an $\epsilon$ such that $t_0+\epsilon\in\mathcal S$. Notice that, since $S^-_{\rho_{t'}}$ forms a foliation with positive mean curvature, $S^-_{\rho_{t_0}}$ is strictly outer-minimizing, then it is easy to see that the strictly outer-minimizing hull $\Sigma_{T_0+t_0}^+$ of $\Sigma_{T_0+t_0}$ is inside $S^-_{\rho_{t_0}}$. Denote $\Sigma'_s$ the weak solution of inverse mean curvature flow with initial data $S^-_{\rho_{t_0}}$ and use this as a barrier for $\Sigma_t$, from Theorem 2.2 in \cite{HI2001}, we know that $\Sigma_{T_0+t_0+s}$ {stays} inside $\Sigma'_s$. Combined with Lemma 2.3 in \cite{HI2001} and comparison principle for smooth inverse mean curvature flow, there is an $\epsilon>0$ such that $\Sigma_{T_0+t_0+s}$ is inside $S_{\rho_{t_0+s}}$ for $s\in[0,\epsilon]$. Therefore, $t_0+\epsilon\in\mathcal S$.
\end{proof}

\begin{corollary}\label{Cor: areapinchsigmat}
Let $T_0$ as above. There are constants $\underline\theta$ and $\bar\theta$ independent of $t$ such that $\Sigma_t$ with $t\geq T_0$ satisfies
\begin{equation}
\underline \theta \sinh^2\underline r_t\leq A(\Sigma_t)\leq \bar\theta \sinh^2\bar r_t.
\end{equation}
\begin{proof}
This follows easily from Lemma \ref{Lem: pinch sigma t} and the fact $A(\Sigma_t)=A(\Sigma_0)e^t$.
\end{proof}
\end{corollary}

In the following, we introduce the interior estimate of mean curvature for smooth inverse mean curvature flow, which is established in \cite{HI2001}. For any $x\in M$, denote
\begin{equation*}
\sigma(x):=\sup\left\{r>0:Rc\geq-\frac{1}{300r^2},\,|\nabla d_x^2|\leq 3d_x,\,|\nabla^2d_x^2|\leq 3\,\,\text{\rm in}\,\,B_r(x)\right\},
\end{equation*}
where $d_x:=\text{\rm dist}(p,x)$. Then we have
\begin{lemma}[see Page 384 \cite{HI2001}]\label{Lem: upper bound mean curvature}
Let $\bar \Sigma_t$ is a smooth inverse mean curvature flow in $M$. For any $x\in\bar\Sigma_t$ and $0<r<\sigma(x)$, we have
\begin{equation}\label{Eq: upper bound mean curvature}
H(x,t)\leq \max\left\{\max_{\bar\Sigma_0\cap B_r(x)}H,\frac{C}{r}\right\},
\end{equation}
where $C$ is a universal constant depending only on the dimension of $M$.
\end{lemma}
\begin{lemma}[see Theorem 3.1 \cite{HI2001}]\label{Lem: weakmeancurvaturebound}
$\Sigma_t$ has uniformly bounded weak mean curvature.
\end{lemma}

We also have the following estimate for the second fundamental form of slices of an inverse mean curvature flow.
\begin{lemma}[see Theorem 5.1 \cite{H2001}]                           
Let $\bar\Sigma_t$ be a smooth inverse mean curvature flow in $(M,g)$. Denote $y_0$ an arbitrary point in $M$ and $\sigma_0$ no greater than the injective radius at $y_0$ such that
\begin{equation}\label{Eq: nearly Euclidean condition}
|\bar\nabla r|\leq 3\sigma_0,\quad \bar \nabla^2 r \leq 3\bar g,\quad \text{\rm in}\quad B_{\sigma_0}(y_0),\quad \text{\rm where} \quad r(x)=\text{\rm dist}(x,y_0)^2 .
\end{equation}
Assuming $\bar\Sigma_t$, $0\leq t<t_0$, has no boundary in $B_{\sigma_0}(y_0)$ and satisfies
\begin{equation}\label{Eq: angle pinch}
0<\beta_1\sigma_0\leq\langle X,\nu\rangle\leq\beta_2\sigma_0,\quad\text{\rm on}\quad \bar\Sigma_t\cap B_{\sigma_0}(y_0),
\end{equation}
for a given smooth vector field $X$. Furthermore
\begin{equation}\label{Eq: Burnstein mean curvature bound}
H_{max}(y_0,\sigma_0)=\sup_{t\geq 0}\sup_{\bar\Sigma_t\cap B_{\sigma_0}(y_0)}H<\infty.
\end{equation}
Then for any $0<\theta<1$, in $\bar\Sigma_t\cap B_{\theta \sigma_0}(y_0)$, we have
\begin{equation}
\lambda_{max}^2\leq C(\beta_1,\beta_2)(1-\theta^2)^2\max\{\sup_{\bar \Sigma_0\cap B_{\sigma_0}(y_0)}\lambda_{max}^2, \sigma_0^{-2}+H_{max}\sigma_0^{-1}+\tilde C)\},
\end{equation}
where $\lambda_{max}$ is the maximum of principle curvature and $\tilde C$ is a universal constant depending on $\beta_1,\beta_2$, $H_{max}$, $(\mathcal L_X g)_{max}$, $(\nabla\mathcal L_X g)_{max}$, $|Rm|_{max}$ and $|\nabla  Rm|_{max}$.
\end{lemma}


\begin{corollary}\label{Cor: interiorestimateh}
There is a $R_0>0$ such that if $\bar\Sigma_t$ is a smooth inverse mean curvature flow enclosing $B_{R_0}$ with $\langle\nu,\frac{\partial}{\partial r}\rangle\geq\delta_0$ on each slice $\bar\Sigma_t$ and $\bar\Sigma_0$ satisfies
\begin{equation}
\max_{\bar\Sigma_0}H\leq C_2\quad\text{\rm and}\quad\max_{\bar\Sigma_0}\|A \|\leq C_3,
\end{equation}
then
\begin{equation}\label{Eq: norm second fundamental form}
\|A\|_{L^\infty(\bar\Sigma_t)}\leq C(\delta_0,C_2,C_3).
\end{equation}
\end{corollary}
\begin{proof}
Let $y_0\in\bar\Sigma_t$, for $R_0$ large enough, we can find a constant $\sigma_0$ such that (\ref{Eq: nearly Euclidean condition}) is true in $B_{\sigma_0}(y_0)$. Let $X=\sigma_0\frac{\partial}{\partial r}$, then (\ref{Eq: angle pinch}) is true for $\beta_1=\delta_0$ and $\beta_2=1$. Furthermore, from Lemma \ref{Lem: upper bound mean curvature}, equation (\ref{Eq: Burnstein mean curvature bound}) is true. Also, it is easy to verify from (\ref{Eq: errorQdecay}) that the ambient curvature $Rm$ and its derivative $\nabla Rm$ are uniformly bounded. Furthermore,
\begin{equation*}
|\mathcal L_X g|=\sigma_0\,|\text{\rm Hess}\,r|\leq C,\quad |\nabla\mathcal L_X g|=\sigma_0|\nabla^3 r|\leq C.
\end{equation*}
Therefore, we have
\begin{equation}
\lambda_{max}\leq C(\delta_0,C_2)\left(1+\sup_{\Sigma_t}\|A_t\|\right)\leq C(\delta_0,C_2,C_3).
\end{equation}
The corollary follows quickly from $H>0$.
\end{proof}

\begin{lemma}[see Corollary 5.6 \cite{H2001}]\label{Lem: C11estimate}
The weak second fundamental form of $\Sigma_t$ satisfies $\|A_t\|\leq C$, where $C$ is a universal constant independent of $t$.
\end{lemma}

For any closed surface $\Sigma$ in $(M,g)$, the Hawking mass is defined to be
\begin{equation}
m_H(\Sigma):=\frac{A(\Sigma)^{\frac{1}{2}}}{(16\pi)^{\frac{3}{2}}}\left(16\pi-\int_{\Sigma}(H^2-4)\,\mathrm d\mu\right).
\end{equation}
As for $\Sigma_t$, we can get a uniform lower bound for the Hawking mass $m_H(\Sigma_t)$, which is independent of $t$.

Before we show the Hawking masses of slices of IMCF has uniform lower bound, we need to show there is an upper bound for the number of components of each slice $\Sigma_t$ of weak inverse mean curvature flow from a connected $C^2$ surface , and  $\Sigma_t$   is connected when $t$ is large enough. The proof of the following lemma even works when the AH manifold $(M^3,g)$ is non orientable.
 The idea follows from Huisken and Ilmanen's work in \cite{HI2001}, where the same issue appears in the analysis for monotonicity of Hawking mass along the weak IMCF. The statement of the lemma is as follows:
\begin{lemma}
Let $\{\Sigma_t\}$ be a weak IMCF from a connected $C^2$ surface $\Sigma$ in asymptotically hyperbolic manifold $(M^3,g)$. Then there is at most $\beta_1+1$ connected components for each $\Sigma_t$, where $\beta_1$ is the number of generators for the fundamental group $\pi_1(M)$. Furthermore, if $\Sigma_t$ is contained in the exterior region, then it is connected.
\end{lemma}
\begin{remark}
Since we can find a compact subset $K$ such that $M-K$ is diffeomorphic to $\mathbb R^3-\bar B_1$, every closed curve can be deformed continuously along a smooth flow into a fixed compact subset of $M$ just as in our proof below. As a result, the fundamental group $\pi_1(M)$ must be finitely generated.
\end{remark}
\begin{proof}
From the proof of \cite[Lemma 4.2]{HI2001}, we may assume $\Sigma_t=\{u=t\}$ and the region outside $\Sigma$ is divided into two connected components $U$ and $V$ by $\Sigma_t$. Assume $\Sigma_t$ has $N$ connected components, labeled by $\Sigma_{t,1},\ldots,\Sigma_{t,N}$. Consider the map
$$
\Phi:\pi_1(M)\to \mathbb Z_2^N,\quad \alpha\mapsto(I_2(\alpha,\Sigma_{t,1}),\ldots,I_2(\alpha,\Sigma_{t,N})),
$$
where $I_2(\alpha,\Sigma_{t,1})$ represents the mod-2 intersection number of $\alpha$ and $\Sigma_{t,i}$. Clearly $\Phi$ is a group homomorphism and so the rank of $Im(\Phi)$ is no greater than $\beta_1$ --- the number of generators for $\pi_1(M)$. Suppose that $N$ is strictly greater than $\beta_1+1$, then we can find $k$ and $l$ such that there is no $\alpha\in \pi_1(M)$ such that
$$
I_2(\alpha,\Sigma_{t,k})=1,\quad I_2(\alpha,\Sigma_{t,l})=1,\quad I_2(\alpha,\Sigma_{t,i})=0,\quad i\neq k,l.
$$
However, due to the connectedness of $U$ and $V$ we can construct a closed curve $\gamma$ such that $\gamma$ crosses $\Sigma_t$ via $\Sigma_{t,k}$ and returns via $\Sigma_{t,l}$, which leads to a contradiction.

Next we prove the second part of the lemma. For asymptotically hyperbolic manifold there is a compact subset $K$ such that $M-K$ is diffeomorphic to $\mathbb R^3-\bar B_1$. Such region is called an exterior region. Assume now $\Sigma_t$ is a slice of weak IMCF contained in the exterior region. As before the region outside $\Sigma$ is divided into two connected component $U$ and $V$ by $\Sigma_t$. If $\Sigma_t$ has more than one components, we take $\gamma$ to be a closed curve crossing through one component and returning through another one. It is easy to see $\gamma$ has nonzero mod-2 intersection number with these components. Since the mod-2 intersection number is invariant under homotopy, we deduce a contradiction by showing that $\gamma$ can be deformed continuously disjoint from $\Sigma_t$. Use $\mathbb R^3-\bar B_1$ as a coordinate chart for $M-K$. Without loss of generality we may assume $\Sigma_t$ is outside $B_3$. Take a smooth cutoff function $0\leq \eta\leq 1$ such that $\eta\equiv 0$ in $(-\infty, 1)$ and $\eta\equiv 1$ in $(2,+\infty)$. Define a smooth vector field
$$
X=-\eta(r)\frac{\partial}{\partial r}
$$
on $\mathbb R^3-\bar B_1$. By zero extension it can be viewed as a global smooth vector field on $M$. Denote $F:M\times \mathbb R\to M$ to be the flow generated by $X$. As time $\tau$ grows, the curve $F(\gamma,\tau)$ will eventually leave $\mathbb R^3-B_3$ and so it will have empty intersection with $\Sigma_t$.
\end{proof}

With above lemma we can obtain a uniform lower bound for Hawking mass.
\begin{lemma}\label{Lem: hawkingmasssigmat}
There is a universal constant $\Lambda$ such that $m_H(\Sigma_t)\geq -\Lambda$ for all $t\geq 0$.
\end{lemma}
\begin{proof}
Let $T_0$ be as in Lemma \ref{Lem: pinch sigma t}, we can choose a fixed $t_0\geq T_0$ such that $\Sigma_{t_0}$ is contained in the exterior region. From Geroch monotonicity formula in \cite[Lemma 5.8]{H2001} we have
\begin{equation*}
m_H(\Sigma_t)\geq m_H(\Sigma_s)+\frac{1}{(16\pi)^{\frac{3}{2}}}\int_s^tA(\Sigma_\tau)^{\frac{1}{2}}\left( 16\pi-8\pi\chi(\Sigma_\tau)+\int_{\Sigma_\tau}(R+6)\,\mathrm d\mu\right)\mathrm d\tau.
\end{equation*}
First let us take $s=0$. With the facts $\chi(\Sigma_\tau)\leq 2(\beta_1+1)$ and $A(\Sigma_\tau)=A(\Sigma_0)e^\tau$, we conclude
\begin{equation*}
\begin{split}
m_H(\Sigma_{t})\geq -\frac{1}{2\pi^{\frac{1}{2}}}\beta_1\left.A(\Sigma_\tau)^{\frac{1}{2}} \right|_{0}^{t_0}-\frac{1}{96\pi^{\frac{3}{2}}}\min_{\bar\Omega_{t_0}}(R+6)_- \left.A(\Sigma_\tau)^{\frac{3}{2}} \right|_{0}^{t_0},\quad \forall\,0\leq t\leq t_0,
\end{split}
\end{equation*}
where $(R+6)_-$ is the negative part of function $R+6$.
Next we take $s=t_0$. Using the facts $\chi(\Sigma_\tau)\leq 2$, $A(\Sigma_\tau)=A(\Sigma_0)e^\tau$ and $R+6=O(e^{-3\alpha})$, we obtain
\begin{equation*}
m_H(\Sigma_t)\geq m_H(\Sigma_{t_0})-CA(\Sigma_0)^{\frac{3}{2}}\frac{2}{\alpha-2} e^{-\frac{\alpha-3}{2}t_0},\quad \forall\, t\geq t_0.
\end{equation*}
This completes the proof.
\end{proof}

The last is a basic regularity result for slices $\Sigma_t$ of the weak solution of an inverse mean curvature flow in $(M,g)$.
\begin{lemma}\label{Lem: W 2,p norm sigma t}
$\Sigma_t$ has uniform $C^{1,\alpha}$ and $W^{2,p}$ norm for any $0<\alpha<1$ and $p>1$.
\end{lemma}

\begin{proof}
The uniform $C^{1,\alpha}$ norm comes from Theorem 1.3 in \cite{HI2001}. Also, from Lemma \ref{Lem: weakmeancurvaturebound}, $\Sigma_t$ has uniformly bounded weak mean curvature. The uniform $W^{2,p}$ norm then follows easily from the theory of elliptic partial differential equations.
\end{proof}

\section{Properties of Approximating Flow}\label{section: approximationflow}

In this section, $(M,g)$ is an asymptotically hyperbolic manifold and $\Sigma_t$ denotes a weak solution of inverse mean curvature flow in $(M,g)$ with connected $C^2$ smooth initial data surface.

A smooth solution of mean curvature flow (MCF) in $(M,g)$ means a smooth family of immersed surfaces $F:\Sigma\times [0,T)\to M$ satisfying
\begin{equation}
\frac{\partial}{\partial t}F(p,t)=-H(p,t)\nu(p,t),\quad F(\cdot,0)=F_0,
\end{equation}
where $F_0$ is a smooth embedding. If we denote slices of a mean curvature flow by $\tilde\Sigma_s=F(\Sigma,s)$, we also call $F:\Sigma\times[0,T)\to M$ a mean curvature flow $\tilde\Sigma_s$ for convenience. Using this statement, from the weak solution $\Sigma_t$, we can construct a mean curvature flow $\Sigma_{t,s}$ as following:
\begin{lemma}\label{Lem: MFCsigmats}
There is a $T_0>0$ such that for any $\Sigma_t$, $t\geq T_0$, there exists a smooth mean curvature flow $\Sigma_{t,s}$, $0<s\leq s_0(t)\leq 1$, such that any sequence $\Sigma_{t,s_j}$ with $s_j\to 0$ has a sub-sequence, which converges  to $\Sigma_t$ in $C^{1,\alpha}$ sense for any $0<\alpha<1$, as $j\to \infty$. In addition, $\Sigma_{t,s}$ has uniformly bounded positive mean curvature and second fundamental from, where the bound is independent of $t$ and $s$.
\end{lemma}

\begin{proof}
The construction is the same as in Lemma 2.6 in \cite{HI2008}, except for some slight differences in calculation caused by evolution equations under mean curvature flow in $(M,g)$. From Lemma \ref{Lem: W 2,p norm sigma t}, we can find a sequence of surfaces $\Sigma^i_t$ such that $\Sigma^i_t$ converges to $\Sigma_t$ in $C^{1,\alpha}$ and $W^{2,p}$ sense for any $0<\alpha<1$ and $p>1$, as $i\to\infty$. Since surfaces $\Sigma^i_t$ are smooth, there exist smooth mean curvature flow $\Sigma^i_{t,s}$ for a short time. Using techniques in \cite{EH1991} to get an interior estimate for second fundamental form of $\Sigma^i_{t,s}$, we know that $\Sigma^i_{t,s}$ exists for a fixed time interval $[0,s_0(t)]$ with $s_0(t)\leq 1$ and $\Sigma^i_{t,s}$ converges smoothly to a limit flow $\Sigma_{t,s}$ with $s\in(0,s_0(t)]$, as $i\to\infty$. Now, we verify the properties stated in the proposition.

From Corollary \ref{Cor: angleestimatesigmat} and the choice of $\Sigma^i_t$, surfaces $\Sigma^i_{t,s}$ can be written as graphs over $\mathbb S^2$. From interior parabolic Schauder regularity theory, the second fundamental form of $\Sigma^i_{t,s}$ satisfies
\begin{equation}\label{2ndfm}
\|A_{t,s}^i\|\leq Cs^{-\frac{1-\beta}{2}},
\end{equation}
where $C$ is a universal constant depending on the $C^{1,\beta}$ norm of $\Sigma^i_t$, but independent of $i$. For any $i\geq 1$ and $p\geq 2$, we compute from the evolution equation
\begin{equation*}
\begin{split}
\frac{\partial}{\partial s}\|A^i_{t,s}\|^2\leq&\Delta_{\Sigma_{t,s}^i}\|A^i_{t,s}\|^2-2\|\nabla_{\Sigma_{t,s}^i}A^i_{t,s}\|^2\\
&\quad+2\|A^i_{t,s}\|^4+4\|A ^i_{t,s}\|^2-4|H^i_{t,s}|^2+(\|A^i_{t,s}\|+\|A^i_{t,s}\|^2) O(e^{-3r})
\end{split}
\end{equation*}
that
\begin{equation*}
\begin{split}
\frac{\partial}{\partial s}\int_{\Sigma^i_{t,s}}\|A^i_{t,s}\|^p\mathrm d\mu\leq & p\int_{\Sigma_{t,s}^i}\|A^i_{t,s}\|^{p+2}\mathrm d\mu+2p\int_{\Sigma_{t,s}^i}\|A^i_{t,s}\|^{p}\mathrm d\mu\\
&\quad+\frac{p}{2}\int_{\Sigma_{t,s}^i}\|A^i_{t,s}\|^{p-1}O(e^{-3r})\mathrm d\mu+\frac{p}{2}\int_{\Sigma_{t,s}^i}\|A^i_{t,s}\|^{p}O(e^{-3r})\mathrm d\mu.
\end{split}
\end{equation*}
For $T_0$ and $i$ large enough, by taking spherical subsolution as barriers and using comparison principle for mean curvature flow, we have that the inner radius of $\Sigma^i_{t,s}$ satisfies $\underline r^i_{t,s}\geq \underline r_t-C$. Here and in the sequel, $C$ is a universal constant independent of $i$ and $s$. Now, applying H\"older's inequality in the last second term and using $O(e^{-3r})$ to handle the area term, then absorbing the last term into the second one, we obtain
\begin{equation*}
\frac{\partial}{\partial s}\int_{\Sigma^i_{t,s}}\|A^i_{t,s}\|^p\mathrm d\mu\leq p\int_{\Sigma^i_{t,s}}\|A^i_{t,s}\|^{p+2}\mathrm d\mu+Cp\int_{\Sigma^i_{t,s}}\|A^i_{t,s}\|^p\mathrm d\mu+Cp\left(\int_{\Sigma^i_{t,s}}\|A^i_{t,s}\|^p\mathrm d\mu\right)^{\frac{p-1}{p}}.
\end{equation*}
Solving this inequality, we have
\begin{equation}\label{Eq: secondfundamentalformsigmatsi}
\left(\int_{\Sigma^i_{t,s}}\|A^i_{t,s}\|^p\mathrm d\mu\right)^{\frac{1}{p}}\leq \left(Cs+\left(\int_{\Sigma^i_{t}}\|A^i_{t}\|^p\mathrm d\mu\right)^{\frac{1}{p}}\right) e^{\frac{1}{\beta}C^2s^\beta+Cs}\leq C
\end{equation}
for any $0<s\leq s_0(t)\leq 1$, where we have used (\ref{2ndfm}). {Taking} $i\to\infty$ and passing above estimate to the limit flow $\Sigma_{t,s}$, we know that $\Sigma_{t,s}$ has a locally uniform $W^{2,p}$ estimate for any $p>2$. Therefore, possibly passing to a subsequence, any sequence of surfaces $\Sigma_{t,s_j}$ with $s_j\to 0$ converges to a limit surface $\Sigma_{t,0}$ in $C^{1,\alpha}$ sense, as $j \to \infty$. Also, {taking} $p\to+\infty$ in (\ref{Eq: secondfundamentalformsigmatsi}), there holds $\|A^i_{t,s}\|\leq C$. Combined with the mean curvature flow equation and $C^{1,\alpha}$ convergence from $\Sigma^i_t$ to $\Sigma_t$, it is not difficult to deduce that $\Sigma_{t,s_j}$ converges to $\Sigma_t$ in $C^0$ sense, which implies $\Sigma_{t,0}$ must coincide with $\Sigma_0$. In addition, by taking $i\to\infty$, $\|A_{t,s}\|\leq C$ follows directly from $\|A^i_{t,s}\|\leq C$.

The bound for mean curvature comes from that of the second fundamental form, then it rests to show that $\Sigma_{t,s}$ has positive mean curvature. Similarly, from the evolution equation of mean curvature, we can calculate that
\begin{equation*}
\int_{\Sigma^i_{t,s}}(H^i_{t,s})_-^2\mathrm d\mu\leq e^{\frac{2}{\beta}C^2s^\beta}\int_{\Sigma^i_t}(H^i_t)_-^2\mathrm d\mu,
\end{equation*}
where $(H^i_{t,s})_-$ and $(H^i_t)_-$ are negative parts of the mean curvatures of $\Sigma^i_{t,s}$ and $\Sigma^i_t$, respectively. {Taking} $i\to\infty$ and passing the estimate to $\Sigma_{t,s}$, we obtain the fact that $H_{t,s}\geq 0$. From the parabolic strong maximum principle, there are only two possible cases: $H_{t,s}>0$ or $\Sigma_{t,s}\equiv\Sigma_t$ is a minimal surface.

We claim that there is a $T_0$ so that $\Sigma_t$ is not a minimal surface for any $t\geq T_0$. Otherwise, there exists a sequence $t_j\to+\infty$ such that $\Sigma_{t_j}$ is a minimal surface. Since a minimal surface possesses higher order estimates from locally uniform $C^{1,\alpha}$ estimate from Lemma \ref{Lem: W 2,p norm sigma t}, in the same spirit of the argument in Proposition \ref{starshapeestimate}, a subsequence of $\Sigma_{t_j}$ will converge to a horosphere in $C^2$ sense, which contradicts to the fact $H\equiv 0$ for all $\Sigma_{t_j}$.
\end{proof}

Denote $\underline r_{t,s}$ and $\bar r_{t,s}$ to be the inner and outer radii of $\Sigma_{t,s}$, respectively. We have the following estimates for surfaces $\Sigma_{t,s}$:

\begin{lemma}\label{Lem: sigmatsestimates}
Possibly decreasing the value of $s_0(t)$, there are constants $\Lambda$, $C_1$, $C_2$, $\underline\theta$ and $\bar\theta$, which are independent of $t$ and $s$, such that any surface $\Sigma_{t,s}$ with $t\geq T_0$ and $0<s\leq s_0(t)$ satisfies $m_H(\Sigma_{t,s})\geq -\Lambda$,
\begin{equation}\label{Eq: sigmatsestimates2}
\underline r_t+\frac{s}{2}-C_1\leq \underline r_{t,s}\leq\bar r_{t,s}\leq \bar r_t+\frac{s}{2}+C_1,
\end{equation}
and
\begin{equation}\label{Eq: sigmatsestimates3}
\underline\theta \sinh^2\underline r_{t,s}\leq A(\Sigma_{t,s})\leq\bar\theta\sinh^2\bar r_{t,s}.
\end{equation}
\end{lemma}
\begin{proof}
From the evolution equation of the mean curvature, by a similar argument as in the proof of Lemma \ref{Lem: MFCsigmats}, we obtain
\begin{equation}\label{Eq: sigmatsestimates1}
\int_{\Sigma_{t,s}}H_{t,s}^2\mathrm \,d\mu\leq e^{\frac{2}{\beta}C^2s^\beta}\int_{\Sigma_t}H_t^2\mathrm d\mu,
\end{equation}
where $C$ is a universal constant independent of $t$ and $s$ and $0<\beta<1$ is a fixed constant. Combined with the definition of Hawking mass and Lemma \ref{Lem: hawkingmasssigmat}, it is easy to see
\begin{equation*}
\liminf_{s\to 0}m_H(\Sigma_{t,s})\geq m_H(\Sigma_t)\geq -\Lambda.
\end{equation*}
With larger $\Lambda$ and smaller $s_0(t)$, $m_H(\Sigma_{t,s})\geq -\Lambda$ for all $t\geq T_0$ and $0<s\leq s_0(t)$. Estimates (\ref{Eq: sigmatsestimates2}) and (\ref{Eq: sigmatsestimates3}) come from the $C^{1,\alpha}$ convergence of $\Sigma_{t,s}$ in Lemma \ref{Lem: MFCsigmats} and Corollary \ref{Cor: areapinchsigmat}.
\end{proof}

Since $\Sigma_{t,s}$ is a smooth surface with positive mean curvature, we can consider the smooth solution of inverse mean curvature flow with initial data surface $\Sigma_{t,s}$. Denote the slice of the solution at time $\tau>0$ by $\bar\Sigma_{t,s,\tau}$ and $\tau_0(t,s)$ the maximum existence time for $\bar\Sigma_{t,s,\tau}$. Let $\underline r_{t,s,\tau}$ and $\bar r_{t,s,\tau}$ be the inner and outer radii of $\bar\Sigma_{t,s,\tau}$.

\begin{lemma}\label{Lem: sigmatstauestimates}
There are constants $\Lambda$ and $C_0$ such that $\bar\Sigma_{t,s,\tau}$ satisfies $m_H(\bar\Sigma_{t,s,\tau})\geq-\Lambda$ and
\begin{equation}\label{Eq: sigmatstauestimates1}
\underline r_{t,s}+\frac{\tau}{2}-C_0\leq \underline r_{t,s,\tau}\leq\bar r_{t,s,\tau}\leq \bar r_{t,s}+\frac{\tau}{2}+C_0.
\end{equation}
\end{lemma}
\begin{proof}
The Hawking mass lower bound comes from Lemma \ref{Lem: sigmatsestimates} and a similar argument in Lemma \ref{Lem: hawkingmasssigmat}. Estimate (\ref{Eq: sigmatstauestimates1}) follows from Lemma \ref{Lem: pinch lemma}.
\end{proof}
\end{appendix}

\Acknowledgements
{The research is supported by the NSFC grants No.11671015 and 11731001. {The authors would like to thank the referees for their careful reading and helpful suggestions on statements of this paper.}}

\end{document}